\documentclass[a4paper,11pt]{amsart}

\usepackage{amsthm}
\usepackage[arrow,curve,matrix]{xy}
\usepackage{amsmath}  
\usepackage{amssymb}
\usepackage{graphicx}
\usepackage{bbm}
\usepackage{mathrsfs}
\usepackage{amscd}
\usepackage[utf8x]{inputenc}

\usepackage{a4wide}

\allowdisplaybreaks

\DeclareMathOperator{\pr}{pr}
\DeclareMathOperator{\id}{id}
\DeclareMathOperator{\sgn}{sgn}
\DeclareMathOperator{\Hom}{Hom}

\DeclareMathOperator{\sHom}{\mathcal{H}\textit{om}}
\DeclareMathOperator{\Ext}{Ext}

\DeclareMathOperator{\Coh}{Coh}

\DeclareMathOperator{\sExt}{\mathcal{E}\textit{xt}}
\DeclareMathOperator{\D}{D}
\DeclareMathOperator{\Ho}{H}
\DeclareMathOperator{\sHo}{\mathcal H}

\DeclareMathOperator{\Hilb}{Hilb}

\DeclareMathOperator{\Aut}{Aut}

\DeclareMathOperator{\supp}{supp}
\DeclareMathOperator{\Pic}{Pic}

\DeclareMathOperator{\triv}{triv}

\DeclareMathOperator{\im}{im}

\DeclareMathOperator{\FM}{FM}
\DeclareMathOperator{\cone}{cone}
\DeclareMathOperator{\pt}{pt}
\DeclareMathOperator{\DAut}{DAut}

\newcommand{\sym}{\mathfrak S}
\newcommand{\C}{\mathbbm C}
\newcommand{\N}{\mathbbm N}

\newcommand{\Z}{\mathbbm Z}
\newcommand{\A}{\mathcal A}
\newcommand{\B}{\mathcal B}
\newcommand{\F}{\mathcal F}
\newcommand{\E}{\mathcal E}

\newcommand{\alt}{\mathfrak a}

\newcommand{\reg}{\mathcal O}
\newcommand{\I}{\mathcal I}

\newcommand{\eps}{\varepsilon}

\renewcommand{\P}{\mathbbm P}
\renewcommand{\theta}{\vartheta}
\renewcommand{\rho}{\varrho}
\renewcommand{\phi}{\varphi}
\renewcommand{\_}{\underline{\,\,\,\,}}

\newtheorem{theorem}{Theorem}[section]
\newtheorem{prop}[theorem]{Proposition}
\newtheorem{lemma}[theorem]{Lemma}
\newtheorem{cor}[theorem]{Corollary}

\theoremstyle{definition}

\newtheorem{example}[theorem]{Example}
\newtheorem{remark}[theorem]{Remark}

\begin{document}
\title[New derived autoequivalences of Hilbert schemes and Kummer varieties]{New derived autoequivalences of Hilbert schemes and generalised Kummer varieties}
\author{Andreas Krug}
\address{Universit\" at Bonn}
\email{akrug@math.uni-bonn.de}
\begin{abstract}
We show that for every smooth projective surface $X$ and every $n\ge 2$ the push-forward along the diagonal embedding gives a 
$\P^{n-1}$-functor into the $\sym_n$-equivariant derived category of $X^n$. Using the Bridgeland--King--Reid--Haiman equivalence this yields a new autoequivalence of the derived category of the Hilbert scheme of $n$ points on $X$. In the case that the canonical bundle of $X$ is trivial and $n=2$ this autoequivalence coincides with the known EZ-spherical twist induced by the boundary of the Hilbert scheme. We also generalise the 16 spherical objects on the Kummer surface given by the exceptional curves to $n^4$ orthogonal $\P^{n-1}$-Objects on the generalised Kummer variety.
\end{abstract}
\maketitle

\section{Introduction}
For every smooth projective surface $X$ over $\C$ and every $n\in\N$ there is the Bridgeland--King--Reid--Haiman equivalence (see \cite{BKR} and \cite{Hai})
\[\Phi\colon \D^b(X^{[n]})\xrightarrow \simeq \D^b_{\sym_n}(X^n)\]
between the bounded derived category of the Hilbert scheme of $n$ points on $X$ and the $\sym_n$-equivariant derived category of the cartesian product of $X$. In \cite{Plo} Ploog used this to give a general construction which associates to every autoequivalence $\Psi\in \Aut(\D^b(X))$ an autoequivalence $\alpha(\Psi)\in\Aut(\D^b(X^{[n]}))$ on the Hilbert scheme.
Recently, Ploog and Sosna \cite{PS} gave a construction that produces out of spherical objects (see \cite{ST}) on the surface $\P^n$-objects (see \cite{HT}) on $X^{[n]}$ which in turn induce further derived autoequivalences.
On the other hand, there are only very few autoequivalences of $\D^b(X^{[n]})$ known to exist independently of $\D^b(X)$:
\begin{itemize}
\item There is always an involution given by tensoring with the alternating representation in $\D^b_{\sym_n}(X^n)$, i.e.\ with the one-dimensional representation on which $\sigma\in \sym_n$ acts via multiplication by $\sgn(\sigma)$. 
\item Addington introduced in \cite{Add} the notion of a $\P^n$-functor generalising the $\P^n$-objects of Huybrechts and Thomas. He showed that for $X$ a K3-surface and $n\ge 2$ the Fourier--Mukai transform $F_a\colon \D^b(X)\to \D^b(X^{[n]})$ induced by the universal sheaf is a $\P^{n-1}$-functor. This yields an autoequivalence of $\D^b(X^{[n]})$ for every K3-surface $X$ and every $n\ge 2$. 
\item For $X=A$ an abelian surface the pull-back along the summation map $\Sigma\colon A^{[n]}\to A$ is a $\P^{n-1}$-functor and thus induces a derived autoequivalence (see \cite{Mea}).
\item The \textit{boundary of the Hilbert scheme} $\partial X^{[n]}$ is the codimension 1 subvariety of points representing non-reduced subschemes of $X$. For $n=2$ it equals $X^{[2]}_\Delta:=\mu^{-1}(\Delta)$ where $\mu\colon X^{[2]}\to S^2X$ denotes the Hilbert-Chow morphism.
For $n=2$ and $X$ a surface with trivial canonical bundle it is known (see \cite[examples 8.49 (iv)]{Huy}) that every line bundle on the boundary of the Hilbert scheme
is an EZ-spherical object (see \cite{Hor}) and thus also induces an autoequivalence. 
We will see in remark \ref{indy} that the induced
automorphisms given by different choices of line bundles on $X^{[2]}_\Delta$ only differ by twists with line bundles on $X^{[2]}$. 
Thus, we will just speak of \textit{the} autoequivalence induced by the boundary referring to the automorphism induced by the EZ-spherical object $\reg_{\mu|X^{[2]}_\Delta}(-2)$.
\end{itemize}
In this article we generalise this last example to surfaces with arbitrary canonical bundle and to arbitrary $n\ge 2$. More precisely, we consider the functor $F\colon \D^b(X)\to \D^b_{\sym_n}(X^n)$ which is defined as the composition of the functor $\triv\colon \D^b(X)\to \D^b_{\sym_n}(X)$ given by equipping every object with the trivial $\sym_n$-linearisation and the push-forward $\delta_*\colon \D^b_{\sym_n}(X)\to \D^b_{\sym_n}(X^n)$ along the diagonal embedding. Then we show in section \ref{dia} the following.
\begin{theorem}\label{main}
For every $n\in \N$ with $n\ge 2$ and every smooth projective surface $X$ 
the functor $F\colon \D^b(X)\to \D^b_{\sym_n}(X^n)$ is a $\P^{n-1}$-functor.
\end{theorem}
In section \ref{BKR} we show that for $n=2$ the induced autoequivalence coincides under $\Phi$ with the autoequivalence induced by the boundary. In section \ref{comp} we compare the autoequivalence induced by $F$ to some other derived autoequivalences of $X^{[n]}$ showing that it differs essentially from the standard autoequivalences and the autoequivalence induced by $F_a$. In particular, the Hilbert scheme always has non-standard autoequivalences even if $X$ is a Fano surface.
In the last section we consider the case that $X=A$ is an abelian surface. We show that after restricting our $\P^{n-1}$-functor from $A^{[n]}$ to the generalised Kummer variety $K_{n-1}A$ it splits into $n^4$ pairwise orthogonal 
$\P^{n-1}$-objects. They generalise the 16 spherical objects on the Kummer surface given by the line bundles $\reg_C(-2)$ on the  exceptional curves.  
\vspace{0.5cm}

\textbf{Acknowledgements}: The author wants to thank Daniel Huybrechts and Ciaran Meachan for helpful discussions.  
This work was supported by the SFB/TR 45 of the DFG (German Research Foundation).
As communicated to the author shortly before he posted this article on the ArXiv, Will Donovan independently discovered the $\P^n$-functor $F$. The author also thanks Nicolas Addington and Will Donovan for pointing out a mistake in the first version of the paper. 
\section{$\P^n$-functors} 
 A \textit{$\P^n$-functor} is defined in \cite{Add} as a functor $F\colon \A\to \B$ of triangulated categories admitting left and right adjoints $L$ and $R$ such that
\begin{enumerate}
 \item There is an autoequivalence $H$ of $\A$ such that 
\[RF\simeq \id\oplus H\oplus H^2\oplus\dots \oplus H^n\,.\]
\item The map 
\[HRF\hookrightarrow RFRF\xrightarrow{R\eps F} RF\]
with $\eps$ being the counit of the adjunction is, when written in the components
\begin{align*}H\oplus H^2\oplus\dots\oplus H^n\oplus H^{n+1}\to \id\oplus H\oplus H^2\oplus\dots\oplus H^n\,,\end{align*}
of the form
\begin{align}\label{matrix}\begin{pmatrix}
  * & * &\cdots &*&*\\
1&*&\cdots&*&*\\
0&1&\cdots&*&*\\
\vdots&\vdots&\ddots&\vdots&\vdots\\
0&0&\cdots&1&* 
  \end{pmatrix}
\,.\end{align}
\item $R\simeq H^nL$. If $\A$ and $\B$ have Serre functors, this is equivalent to $S_\B FH^n\simeq FS_\A$.
\end{enumerate}
In the following we always consider the case that $\A$ and $\B$ are (equivariant) derived categories of smooth projective varieties and $F$ is a Fourier--Mukai transform. 
The \textit{$\P^n$-twist} associated to a $\P^n$-functor $F$ is defined as the double cone 
\[P_F:=\cone\left(\cone(FHR\to FR)\to \id  \right)\,.\]
The map defining the inner cone is given by the composition 
\[FHR\xrightarrow{FjR}FRFR\xrightarrow{\eps FR-FR\eps}FR\]
where $j$ is the inclusion given by the decomposition in (i). The map defining the outer cone is induced by the counit
$\eps\colon FR\to \id$ (for details see \cite{Add}).
Taking the cones of the Fourier--Mukai transforms indeed makes sense, since all the occurring maps are induced by maps between the integral kernels (see \cite{AL}). 
We set $\ker R:=\{B\in \B\mid RB=0\}$. By the adjoint property it equals the right-orthogonal complement 
$(\im F)^\perp$.
\begin{prop}[{\cite[section 3]{Add}}]\label{twistaction}
Let $F\colon \A\to \B$ be a $\P^n$-functor.
\begin{enumerate}
\item We have $P_F(B)=B$ for $B\in \ker R$.
\item $P_F\circ F\simeq F\circ H^{n+1}[2]$.
\item The objects in $\im F\cup \ker R$ form a spanning class of $\B$.
\item $P_F$ is an autoequivalence.
\end{enumerate}
\end{prop}
\begin{example}
\begin{enumerate}\label{ex}
 \item Let $\B=\D^b(X)$ for a smooth projective variety $X$. A \textit{$\P^n$-object} (see \cite{HT}) is an object $E\in \B$ 
such that $E\otimes \omega_X\simeq E$ and $\Ext^*(E,E)\cong \Ho^*(\P^n,\C)$ as $\C$-algebras (the ring structure on the left-hand side is the Yoneda product and on the right-hand side the cup product). 
A $\P^n$-object can be identified with the $\P^n$-functor
\[F\colon \D^b(\pt)\to \B\quad,\quad \C\mapsto E\]
with $H=[-2]$. Note that the right adjoint is indeed given by $R=\Ext^*(E,\_)$. The $\P^n$-twist associated to the functor $F$ is the same as the $\P^n$-twist associated to the object $E$ as defined in \cite{HT}.
\item A $\P^1$-functor is the same as a \textit{spherical functor} (see \cite{Ann}) where the unit 
\[\id\xrightarrow{\eta} RF\to H\]
splits. In this case there is also the \textit{spherical twist} given by 
\[T_F:=\cone\left(FR\xrightarrow {\eps} \id\right)\,.\]
It is again an autoequivalence with $T_F^2=P_F$ (see \cite[p. 33]{Add}).  
\end{enumerate}
\end{example}
\begin{lemma}\label{relation}
\begin{enumerate}
 \item Let $\Psi\in \Aut(\A)$ such that $\Psi\circ H\simeq H\circ \Psi$. Then $F\circ \Psi$ is again a $\P^n$-functor with the property
\[P_{F\circ \Psi}\simeq P_F\,.\]
\item Let $\Phi\colon \B\to \mathcal C$ be an equivalence of triangulated categories. Then $\Phi\circ F$ is again a $\P^n$-functor with the property that 
\[P_{\Phi\circ F}\circ \Phi\simeq \Phi\circ P_F\,.\]
\end{enumerate}
\end{lemma}
\begin{proof}
The proof is analogous to the proof of the corresponding statement for spherical functors \cite[Proposition 2]{Ann}.
\end{proof}
\begin{cor}\label{pnobj}
Let $E_1,\dots,E_n\in \B$ be a collection of pairwise orthogonal (that means $\Hom^*(E_i,E_j)=0=\Hom^*(E_j,E_i)$ for $i\neq j$) $\P^n$-objects with associated twists $p_i:=P_{E_i}$. Then 
\[\Z^n\to \Aut(\A)\quad, \quad (\lambda_1,\dots,\lambda_n)\mapsto p_1^{\lambda_1}\circ\dots\circ p_n^{\lambda_n}\]
defines a group isomorphism $\Z^n\cong \langle p_1,\dots,p_n\rangle\subset \Aut(\B)$.  
\end{cor}
\begin{proof}
By part (ii) of the previous lemma the $p_i$ commute which means that the map is indeed a group homomorphism onto 
the subgroup generated by the $p_i$. Let $g=p_1^{\lambda_1}\circ \dots \circ p_n^{\lambda_n}$. Then $g(E_i)=E_i[2n\lambda_i]$ by proposition \ref{twistaction}. Thus, $g=\id$ implies $\lambda_1=\dots=\lambda_n=0$.  
\end{proof}
\begin{lemma}
 Let $X$ be a smooth variety, $T\in \Aut(\D^b(X))$, and $A,B\in \D^b(X)$ objects such that $TA=A[i]$ and $TB=B[j]$ for some $i\neq j\in \Z$. Then $A\perp B$ and $B\perp A$.
\end{lemma}
\begin{proof}
 See \cite[p. 11]{Add}.
\end{proof}
\begin{remark}\label{noshift}
 This shows together with proposition \ref{twistaction} that for a $\P^n$-functor $F$ with $H=[-\ell]$ for some $\ell\in \Z$ there does not exist a non zero-object $A$ with $T_F(A)=A[i]$ for any values of $i$ besides $0$ and $-n\ell+2$ because such an object would be orthogonal to the spanning class $\im F\cup \ker R$.
\end{remark}
\section{The diagonal embedding}\label{dia}
Let $X$ be a smooth projective surface over $\C$ and $2\le n\in \N$. We denote by 
$\delta\colon X\to X^n$ the diagonal embedding. We want to show that $F\colon \D^b(X)\to \D^b_{\sym_n}(X^n)$ given as the composition
\[\D^b(X)\xrightarrow{\triv} \D^b_{\sym_n}(X)\xrightarrow{\delta_*} \D^b_{\sym_n}(X^n)\]
of the functor which maps each object to itself equipped with the trivial action and the equivariant push-forward is a $\P^{n-1}$-functor.
Its right adjoint $R$ is given as the composition
\[\D^b_{\sym_n}(X^n)\xrightarrow{\delta^!} \D^b_{\sym_n}(X)\xrightarrow{[\_]^{\sym_n}} \D^b(X)\]
of the usual right adjoint (see \cite{Has} for equivariant Grothendieck duality) and the functor of taking invariants.
We consider the standard representation $\rho$ of $\sym_n$ as the quotient of the regular representation $\C^n$ by the one dimensional invariant subspace.
The normal bundle sequence 
\[0\to T_X\to T_{X^n\mid X}\to N\to 0\]
where $N:=N_\delta=N_{X/X^n}$ is of the form
\[0\to T_X\to T_X^{\oplus n}\to N\to 0\] 
where the map $T_X\to T_X^{\oplus n}$ is the diagonal embedding. When considering $T_{X^n|X}$ as a ${\sym_n}$-sheaf equipped with the natural linearisation it is given by $T_X\otimes \C^n$ where $\C^n$ is the regular representation. Thus, as a $\sym_n$-sheaf, the normal bundle equals $T_X\otimes \rho$. We also see that the normal bundle sequence splits using e.g.\ the splitting
\[T_X\otimes \C^n\to T_X\quad,\quad (v_1,\dots,v_n)\mapsto \frac 1n (v_1+\dots+v_n)\,.\]
\begin{theorem}[\cite{AC}]
 Let $\iota\colon Z\hookrightarrow M$ be a regular embedding of codimension $c$ such that the normal bundle sequence splits. Then there is an isomorphism 
\begin{align}\label{cald}\iota^*\iota_*(\_)\simeq (\_)\otimes(\bigoplus_{i=0}^c \wedge^i N_{Z/M}^\vee [i])\end{align}
of endofunctors of $\D^b(Z)$.
\end{theorem}
\begin{cor}
Under the same assumptions, there is an isomorphism
\begin{align}\label{caldcor}\iota^!\iota_*(\_)\simeq (\_)\otimes(\bigoplus_{i=0}^c \wedge^i N_{Z/M} [-i])\end{align}
\end{cor}
\begin{proof}
Tensorise both sides of (\ref{cald}) by $\iota^!\reg_M\simeq \wedge^c N_{Z/M}[-c]$.
\end{proof}
In the case that $\iota=\delta$ from above this yields the isomorphism of functors
\begin{align}\label{before}\delta^!\delta_*(\_)\simeq (\_)\otimes (\bigoplus_{i=0}^{2(n-1)}\wedge^i(T_X\otimes \rho)[-i])\,.\end{align}
\begin{lemma}\label{monadmulti}
For $0\le i,j\le c$ with $i+j\le c$ the component \[(\_)\otimes \wedge^i N_{Z/M}\otimes \wedge^j N_{Z/M}\to (\_)\otimes \wedge^{i+j} N_{Z/M}\] 
of the monad multiplication $\iota^!\eps\iota_*\colon \iota^!\iota_*\iota^!\iota_*\to \iota^!\iota_*$ is given by the wedge pairing.
\end{lemma}
\begin{proof}
For $E\in \D^b(M)$ the object $\iota^! E$ can be identified with $\sHom(\iota_*\reg_Z,E)$ considered as an object in 
$\D^b(Z)$. Under this identification the counit map $\sHom(\iota_*\reg_Z,E)\to E$ is given by the evaluation $\phi\mapsto \phi(1)$ (see \cite[section III.6]{Har1}). 
Now we get for $F\in \D^b(Z)$ the identifications
\[\iota^!\iota_*\iota^!\iota_*F\simeq \sHom(\iota_*\reg_Z,\iota_*\reg_Z)\otimes_{\reg_Z} \sHom(\iota_*\reg_Z,\iota_*\reg_Z)\otimes_{\reg_Z} F \]
and $\iota^!\iota_* F\simeq \sHom(\iota_*\reg_Z,\iota_*\reg_Z)\otimes_{\reg_Z} F$. Under these identifications the component
\[\sExt^i(\iota_*\reg_Z,\iota_*\reg_Z)\otimes_{\reg_Z} \sExt^j(\iota_*\reg_Z,\iota_*\reg_Z)\otimes_{\reg_Z}F\to \sExt^{i+j}(\iota_*\reg_Z,\iota_*\reg_Z)\otimes_{\reg_Z}F\]
of the monad multiplication equals the Yoneda product. The Yoneda product corresponds to the wedge product under the isomorphism $\sExt^i(\reg_Z,\reg_Z)\cong \wedge^iN_{Z/M}$ (see \cite[p. 442]{Has}).
\end{proof}
\begin{lemma}[{\cite[Appendix B]{Sca1}}]\label{Scalemma}
Let $V$ be a two-dimensional vector space with a basis consisting of vectors $u$ and $v$. Then the space of invariants $[\wedge^i(V\otimes \rho)]^{\sym_n}$ is one-dimensional if $0\le i \le 2(n-1)$ is even and zero if it is odd. In the  
even case $i=2\ell$ the space of invariants is spanned by the image of the vector $\omega^\ell$, where 
\[\omega=\sum_{i=1}^k ue_i\wedge ve_i \in \wedge^2(V\otimes \C^n)\,,\]
under the projection induced by the projection $\C^n\to \rho$.
\end{lemma} 
\begin{cor}\label{wedgeinva}
For a vector bundle $E$ on $X$ of rank two and $0\le \ell\le n-1$ there is an isomorphism
\[ [\wedge^{2\ell}(E\otimes \rho)]^{\sym_n}\cong (\wedge^2 E)^{\otimes \ell}\,.\]
\end{cor}
\begin{proof}
 The isomorphism is given by composing the morphism
\[(\wedge^2E)^{\otimes \ell}\to \wedge^\ell(E\otimes \C^n)\quad,\quad x_1\otimes \dots\otimes x_\ell\mapsto \sum_{1\le i_1<\dots<i_\ell\le n} x_1 e_{i_1}\wedge\dots\wedge x_\ell e_{i_{\ell}}\]
with the projection induced by the projection $\C^n\to \rho$.
\end{proof}
We set $H:=\wedge^2 T_X[-2]=\omega_X^\vee[-2]=S_X^{-1}$ as the inverse of the Serre functor on $X$. 
\begin{cor}
There is the isomorphism of functors
\[RF\simeq \id\oplus H\oplus H^2\oplus\dots\oplus H^{n-1}\,.\] 
\end{cor}
\begin{proof}
This follows by formula (\ref{before}) and corollary \ref{wedgeinva}.
\end{proof}

\begin{lemma}
The map $HRF\to RF$ defined in condition (ii) for $\P$-functors is for this pair $F\rightleftharpoons R$ given by a matrix of the form (\ref{matrix}).
\end{lemma}
\begin{proof}
The generators $\omega^\ell$ from lemma \ref{Scalemma} are mapped to each other by wedge product. By lemma \ref{monadmulti} the components $H\circ H^k\to H^{k+1}$ for $k=0,\dots,n-2$ of the monad multiplication are given by wedge product. Hence, they are isomorphisms.
\end{proof}
\begin{lemma}
There is the isomorphism $S_{X^n} F H^{n-1}\simeq F S_X$. 
\end{lemma}
\begin{proof}
 For $\E\in \D^b(X)$ there are natural isomorphisms
\begin{align*}
 S_{X^n}F H^{n-1}(\E)= \omega_{X^n}[2n]\otimes \delta_*(\E\otimes \omega_X^{-(n-1)}[-2(n-1)])&\simeq
\omega_{X}^{\boxtimes n}\otimes \delta_*(\E\otimes \omega_X^{-(n-1)})[2]\\
&\overset{\text{PF}}\simeq \delta_*(\E\otimes \omega_X [2])=FS_X(\E)\,.
\end{align*}
\end{proof}
All this together shows theorem \ref{main}, i.e.\ that $F=\delta_*\circ \triv$ is indeed a $\P^{n-1}$-functor.
\section{Composition with the Bridgeland--King--Reid--Haiman equivalence}\label{BKR}
The \textit{isospectral Hilbert scheme} $I^nX\subset X^{[n]}\times X^n$ is defined as the reduced fibre product
$I^nX:=\left( X^{[n]}\times_{S^nX} X^n\right)_{\text{red}}$ with the defining morphisms being the Hilbert--Chow morphism $\mu\colon X^{[n]}\to S^nX$ and the quotient morphism $\pi\colon X^n\to S^nX$.
Thus, there is the commutative diagram
\[
\begin{CD}
I^nX
@>{p}>>
X^n \\
@V{q}VV
@VV{\pi}V \\
X^{[n]}
@>>{\mu}>
S^nX\,.
\end{CD} 
\] 
 The \textit{Bridgeland--King--Reid--Haiman equivalence} is the functor
\[\Phi:=\FM_{\reg_{I^nX}}\circ \triv=p_*\circ q^*\circ \triv\colon \D^b(X^{[n]})\longrightarrow \D^b_{\sym_n}(X^n)\,.\]
By the results in \cite{BKR} and \cite{Hai} it is indeed an equivalence. The isospectral Hilbert scheme can be identified with the blow-up of $X^n$ along the union of all the pairwise diagonals $\Delta_{ij}=\{(x_1,\dots,x_n)\in X^n\mid x_i=x_j\}$ (see \cite{Hai}).
By lemma \ref{relation} the functor composition $\Phi^{-1}\circ F\colon \D^b(X)\to\D^b(X^{[n]})$ is again a $\P^{n-1}$-functor  and thus yields an autoequivalence of the derived category of the Hilbert scheme. We consider in this section the case $n=2$.
Then $I^2X$ equals the blow-up of $X^2$ in the diagonal $\Delta$. In particular, $I^2X$ is smooth. We denote the exceptional divisor by $E$ and the inclusion of $I^2X$ into $X^{[2]}\times X^2$ by $j$.
Let $\tau=(1\,\,2)$ denote the non-trivial element of $\sym_2$.
\begin{lemma}
The functor $\Phi^{-1}\colon \D^b_{\sym_2}(X^2)$ is given by the composition $[\_]^{\sym_2}\circ \FM_{\mathcal Q}$ with 
\[\mathcal Q=j_*\reg(E)\,\in \D^b_{\sym_2}(X^2\times X^{[2]})\,.\]
The $\sym_2$-linearisation of $\mathcal Q$ restricts on $E$ to $\tau$ acting by $-1$ on $\reg_E(E)$.   
\end{lemma}
\begin{proof}
By the general formula for the right-adjoint of a Fourier--Mukai transform (see e.g.\ \cite[Proposition 5.9]{Huy}), we have
$\mathcal Q=\reg_{I^2X}^\vee\otimes q^*\omega_{X^{[2]}}[4]$. 
The canonical bundle of the blow-up is given by
$\omega_{I^2X}\cong p^*\omega_{X^2}\otimes \reg(E)$. Let $N$ be the normal bundle of the codimension 4 regular embedding $I^2X\to X^{[2]}\times X^2$. By adjunction formula
\[\wedge^4N\cong\omega^\vee_{X^{[2]}\times X^2\mid I^2X}\otimes \omega_{I^2X}\cong q^*\omega^\vee_{X^{[2]}}\otimes \reg(E)\,.\]
It follows by Grothendieck-Verdier duality for regular embeddings that
\[\mathcal Q=\reg_{I^2X}^\vee\otimes q^* \omega_{X^{[2]}}[4]\simeq \wedge^4 N[-4]\otimes q^*\omega_{X^{[2]}}[4]\simeq \reg(E)\,.\]
Let us assume that $\tau$ acts trivially on $\mathcal Q_{|E}=\reg_E(E)$. Then by \cite[Theorem 2.3]{DN} the sheaf of invariants 
$q_*\mathcal \mathcal Q^{\sym_2}$ is the descent of $\mathcal Q$, i.e.\ $q^*q_* \mathcal Q^{\sym_2}=\mathcal Q$. We also have $\Phi^{-1}(\reg_{X^n})= q_* \mathcal Q^{\sym_2}$ and by \cite[Proposition 1.3.3]{Sca1} $\Phi^{-1}(\reg_{X^n})=\reg_{X^{[n]}}$. All this together implies that $\reg(E)$ is the trivial line bundle which is a contradiction.
\end{proof}
If the surface $X$ has trivial canonical bundle, it is known that any line bundle $L$ on the boundary $\partial X^{[2]}=X^{[2]}_{\Delta}$ of the 
Hilbert scheme of two points on $X$ is an EZ-spherical object (see \cite[examples 8.49 (iv)]{Huy}).
That means that the functor
\[\tilde F_L\colon \D^b(X)\to \D^b(X^{[2]})\quad ,\quad \E\mapsto j_*(L\otimes \mu_\Delta^*E)\]
is a spherical functor where the maps $j$ and $\mu_\Delta$ come from the fibre diagram
\[
\begin{CD}
X^{[2]}_{\Delta}
@>{j}>>
X^{[2]} \\
@V{\mu_\Delta}VV
@VV{\mu}V \\
X
@>>{d}>
S^2X
\end{CD} 
\] 
with $d$ being the diagonal embedding. The map $\mu_\Delta\colon X^{[2]}_\Delta\to X$ equals the $\P^1$-bundle 
$\P(\Omega_X)\to X$. 
We need the following slight generalisation of \cite[Proposition 11.12]{Huy} for a blow-up
\[
\begin{CD}
E
@>{i}>>
\tilde X \\
@V{\pi}VV
@VV{p}V \\
Y
@>>{j}>
X
\end{CD} 
\]
of a smooth projective variety $X$ along a smooth subvariety $Y$ of codimension $c$. 
\begin{lemma}
For every $\F\in \Coh(Y)$ and every $k\in \Z$ there is an isomorphism
\[\mathcal H^k(p^*j_*\F)\cong i_*\left( \pi^* \F\otimes \wedge^{-k}\Omega_\pi\otimes \reg_\pi(-k)\right)\,.\]
\end{lemma}
\begin{proof}
This can be proven locally. Hence, we may assume that $Y=Z(s)$ is the zero locus of a global section of a vector bundle $\E$ of rank $c$.
Thus, the blow-up diagram can be enlarged to
\[
\begin{CD}
E
@>{i}>>
\tilde X @>{\iota}>> \P(\E)\\
@V{\pi}VV
@VV{p}V @VV{g}V\\
Y
@>>{j}>
X @>>{\id}> X
\end{CD} 
\]
where $\iota$ is a closed embedding of codimension $c-1$ such that the normal bundle $M$ has the property $\wedge^k M^\vee_{|E}=\wedge^{-k}\Omega_\pi\otimes \reg_\pi(-k)$ (see \cite[p. 252]{Huy}). The outer square is a flat base change. It follows that
\begin{align}\label{eq}
\iota_*p^*j_*\F\simeq \iota_*\iota^*g^*j_*\F\simeq \iota_*\iota^*\iota_*i_*\pi^*\F\simeq \iota_*i_*(\pi^*\F\otimes i^*\iota^*\iota_*\reg_{\tilde X}) 
\end{align}
where the last isomorphism is given by applying the projection formula two times. Now $\mathcal H^k(\iota^*\iota_* \reg_{\tilde X})\cong \wedge^{-k} M^\vee$ is locally free for every $k\in \Z$. It follows that
\[\mathcal H^k(\pi^* \F\otimes i^*\iota^*\iota_*\reg_{\tilde X})\cong \pi^* \F\otimes i^*\mathcal H^k(\iota^*\iota_*\reg_{\tilde X})\cong \pi^*\F\otimes \wedge^k M^\vee_{|E} \cong  \pi^*\F\otimes \wedge^{-k}\Omega_\pi\otimes \reg_\pi(-k)\,.\] 
By (\ref{eq}) also 
\[\iota_*\mathcal H^k(p^*j_*\F)\cong \mathcal H^k(\iota_* p^*j_*\F)\cong \mathcal H^k\left(\iota_*i_*(\pi^* \F\otimes i^*\iota^*\iota_*\reg_{\tilde X})\right)\cong \mathcal \iota_*i_*\mathcal H^k(\pi^* \F\otimes i^*\iota^*\iota_*\reg_{\tilde X})
\]
which proves the assertion since $\iota_*\colon \Coh(\tilde X)\to \Coh(\P(\E))$ is fully faithful.
\end{proof}
\begin{remark}
If $X$ carries an action by a finite group $G$ and $Y$ is invariant under this action, $G$ also acts on the blow-up $\tilde X$.
The bundle $M_{|E}$ of the proof is a quotient of the normal bundle $N_{E/\P(\E)}\cong \pi^* N_{Y/X}$. 
In the case that there is a group action this quotient is $G$-equivariant. Thus, the formula of the lemma is in this case also true for $\F\in \Coh_G(X)$ with the action on the right hand side induced by the linearisation of the wedge powers of $M$ respectively $N_{Y/X}$. 
\end{remark}
\begin{prop}\label{underbkr}
Let $X$ be a smooth projective surface (with arbitrary canonical bundle). Then there is an isomorphism of 
functors $\Phi^{-1}\circ F\cong \tilde F_{\reg_{\mu_{\Delta}}(-2)}[1]$, where $F=\delta_*\circ \triv$.  
\end{prop}
\begin{proof}
The $\P^1$-bundles $p_{\Delta}\colon E\to X$ and $\mu_{\Delta}\colon X^{[2]}_{\Delta}\to X$ are isomorphic via $q\colon I^2X\to X^{[2]}$. 
The centre $\Delta$ of the blow-up $p\colon I^2X\to X^2$  has codimension 2 in $X^2$. Thus, $p^*\delta_* \F$ 
is cohomologically concentrated in degree $0$ and $-1$. Since $p^*\delta_* \F$ is concentrated on $E$ where $\sym_2$ is acting trivially, one can take the invariants even before applying the push-forward along $q\colon I^2X\to X^{[2]}$. The group $\sym_2$ acts on $\wedge^0 N_{\Delta/X^2}$ trivially and on $N_{\Delta/X^2}$ alternating. By the previous lemmas and the remark it follows that there is a natural isomorphism $\Phi^{-1}\circ F(A)\simeq \tilde F_{\reg_{\mu_{\Delta}}(-2)}(A)[1]$ for $A\in \Coh(X)$.
Using the spectral sequence associated to the pull-back $q^*$ yields the isomorphism of functors.
\end{proof} 
\begin{remark} The proposition says in particular that $\tilde F_{\reg_{\mu_{\Delta}}(-2)}$ is also a spherical functor in the case that $\omega_X$ is not trivial. One can also prove this directly and for general $L$ instead of $\reg_{\mu_{\Delta}}(-2)$.
\end{remark}
\begin{remark}
Let $\hat F:=M_{\alt}\circ F\colon \D^b(X)\to \D^b(X^2)$ where $M_\alt$ is the tensor product with the alternating representation $\alt$. It is again a spherical functor by lemma \ref{relation}. Following the proof of proposition \ref{underbkr} one gets that $\Phi^{-1}\circ \hat F\cong \tilde F_{\reg_{\mu_{\Delta}}(-1)}$.  
\end{remark}
\begin{remark}\label{indy}
Since $X^{[2]}_\Delta$ equals the $\P^1$-bundle $\P(\Omega_X)$ over $X$, its canonical bundle is given by $\omega_{X^{[2]}_{\Delta}}=\mu_\Delta^*\omega_X^2\otimes\reg_{\mu_{\Delta}}(-2)$ (see e.g.\ \cite[B 5.8]{Ful}). The Hilbert-Chow morphism $\mu$ is a crepant resolution, i.e.\ $\omega_{X^{[2]}}\cong \mu^* \omega_{S^2X}$. Thus,
\[\omega_{X^{[2]}|X^{[2]}_\Delta}\cong \mu_\Delta^*(\omega_{S^2X\mid \Delta})\cong \mu_\Delta^*\omega_X^2\,.\]   
It follows by adjunction formula that $\reg_{X^{[n]}_\Delta}(X^{[n]}_\Delta)=\reg_{\mu_\Delta}(-2)$. The line bundle $\reg_{X^{[2]}}(-X^{[2]}_\Delta)$ has a square root $D:=\reg(-X^{[2]}_\Delta/2)\in \Pic X^{[2]}$ (see \cite[lemma 3.8]{Leh}). Its restriction to $X^{[2]}_\Delta$ is of the form $D_{|X^{[2]}_\Delta}=\reg_{\mu_\Delta}(1)$. Using this, we can rewrite for a general line bundle $L=\mu_\Delta^* K\otimes \reg_{\mu_\Delta}(i)\in \Pic X^{[2]}_\Delta$ the spherical functor $\tilde F_L$ as $\tilde F_L=M_D^{i+2}\circ \tilde F_{\reg_{\mu_\Delta(-2)}}\circ M_K$ where $M_K$ denotes the autoequivalence given by tensor product with $K$.
The analogue of lemma \ref{relation} for spherical functors thus yields $t_L=M_D^{i+2}\circ t_{\reg_{\mu_\Delta}(-2)}\circ M_D^{-(i+2)}$ where $t_L$ is the spherical twist associated to $\tilde F_L$.    
\end{remark}
\begin{remark}\label{supp}
For general $n\ge 2$ every object in the image of $\Phi^{-1}\circ F$ still is supported on $X^{[n]}_\Delta=\mu^{-1}(\Delta)$. 
\end{remark}
\section{Comparison with other autoequivalences}\label{comp}
In the following we will denote the $\P^n$-twist associated to $F$ respectively $\Phi^{-1}\circ F$ by $b\in\Aut(\D^b_{\sym_n}(X^{n}))\cong \Aut(\D^b(X^{[n]}))$. In the case that $n=2$ the functor $F$ is spherical (see example \ref{ex} (ii)). We denote the associated spherical twist by $\sqrt b$.
\begin{prop}
The automorphism $b$ is not contained in the group of standard automorphisms
\[\Aut(\D^b(X^{[n]}))\supset \DAut_{st}(X^{[n]})\cong \Z\times \left(\Aut(X^{[n]})\ltimes \Pic(X^{[n]})\right)\]
generated by shifts, push-forwards along automorphisms and taking tensor products by line bundles. The same holds in the case $n=2$ for $\sqrt b$.  
\end{prop}
\begin{proof}
Let $[\xi]\in X^{[n]}\setminus X^{[n]}_\Delta$, i.e.\ $|\supp \xi|\ge 2$. Then by remark \ref{supp} and proposition \ref{twistaction} (i), we have $b(\C([\xi]))= \C([\xi])$. Let $g=[\ell]\circ \phi_*\circ M_L\in \DAut_{st}(X^{[n]})$ where 
$M_L$ is the functor $E\mapsto E\otimes L$ for an $L\in \Pic X^{[n]}$. Then $g(\C([\xi]))=\C(\phi([\xi]))[\ell]$. Thus, the assumption $b=g$ implies $\ell=0$ and also $\phi=\id$, since $X^{[n]}\setminus X^{[n]}_\Delta$ is open in $X^{[n]}$.
Thus, the only possibility left for $b\in \DAut_{st}(X^{[n]})$ is $b=M_L$ for some line bundle $L$ which can not hold by 
proposition \ref{twistaction} (ii). The proof that $\sqrt b\notin \DAut_{st}(X^{[n]})$ is the same.    
\end{proof}
In \cite{Plo} Ploog gave a general construction which associates to derived autoequivalences of the surface $X$ derived autoequivalences of the Hilbert scheme $X^{[n]}$.
Let $\Psi\in \Aut(\D^b(X))$ with Fourier--Mukai kernel $\mathcal P\in \D^b(X\times X)$. The object 
$\mathcal P^{\boxtimes n}\in \D^b(X^n\times X^n)$ carries a natural $\sym_n$-linearisation given by permutation of the box factors. Thus, it induces a $\sym_n$-equivariant derived autoequivalence $\alpha(\Psi):=\FM_{\mathcal P^{\boxtimes n}}$ of 
$X^n$. This gives the following.
\begin{theorem}[\cite{Plo}]\label{Plo}
The above construction gives an injective group homomorphism 
\[\alpha\colon \Aut(\D^b(X))\to\Aut(\D^b_{\sym_n}(X^n))\cong \Aut(\D^b(X^{[n]}))\,.\]  
\end{theorem}
\begin{remark}
For every $\phi\in \Aut(X)$ we have $\alpha(\phi_*)=(\phi^n)_*$ where $\phi^n$ is the $\sym_n$-equivariant automorphism of $X^n$ given by $\phi(x_1,\dots,x_n)=(\phi(x_1),\dots,\phi(x_n))$. 
Furthermore, $\phi$ acts on $X^{[n]}$ by the morphism $\phi^{[n]}$, which is given by $\phi^{[n]}([\xi])=[\phi(\xi)]$, and on 
$X^n$ by the morphism $\phi^n$. The isospectral Hilbert scheme $I^nX\subset X^{[n]}\times X^n$ is invariant under the 
induced action of $\Aut(X)$ on $X^{[n]}\times X^n$. 
Thus, the Bridgeland--King--Reid--Haiman equivalence $\Phi=\FM_{\reg_{I^nX}}$ is $\Aut(X)$-equivariant in the sense that $\Phi\circ (\phi^{[n]})_*\simeq (\phi^n)_*\circ \Phi$. 
Hence, $\alpha(\phi_*)\in \Aut(\D^b_{\sym_n}(X^n))$ corresponds to $\phi^{[n]}_*\in \Aut(\D^b(X^{[n]}))$.
For $L\in \Pic X$ we have $\alpha(M_L)=M_{L^{\boxtimes n}}$ where $L^{\boxtimes n}$ is considered as a $\sym_n$-equivariant line bundle with the natural linearisation. Under $\Phi$, the automorphism $M_{L^{\boxtimes n}}$ corresponds to $M_{\mathcal D_L}\in \Aut(\D^b(X^{[n]}))$ where $\mathcal D_L\in \Pic X^{[n]}$ is the line bundle $\mathcal D_L:=\mu^*((L^{\boxtimes n})^{\sym_n})$ 
(see e.g.\ \cite[lemma 9.2]{Kru2}).    
\end{remark}
\begin{lemma}
 \begin{enumerate}
  \item For every automorphism $\phi\in \Aut(X)$ we have $b\circ \alpha(\phi_*)=\alpha(\phi_*)\circ b$ and for $n=2$ also 
$\sqrt b\circ \alpha(\phi_*)= \alpha(\phi_*)\circ \sqrt b$.
  \item For every line bundle $L\in \Pic(X)$ we have $b\circ \alpha(M_L)=\alpha(M_L)\circ b$ and for $n=2$ also 
$\sqrt b\circ \alpha(M_L)= \alpha(M_L)\circ \sqrt b$.
 \end{enumerate}
\end{lemma}
\begin{proof} 
 We have $\alpha(\phi_*)\circ F\simeq F\circ \phi_*$ and $\alpha(M_L)\circ F\simeq F\circ M_L^n$. The assertions now follow by lemma \ref{relation} (for $\sqrt b$ one has to use the analogous result \cite[proposition 2]{Ann} for spherical twists). 
\end{proof}
Let $G\subset \Aut(\D^b(X^{[n]}))$ be the subgroup generated by $b$, shifts, and $\alpha (\DAut_{st}(X))$.
\begin{prop}
 The map
\[S\colon \Z\times \Z\times (\Aut(X)\ltimes \Pic(X))\to \Aut(\D^b_{\sym_n}(X^n))\quad,\quad (k,\ell,\Psi)\mapsto b^k\circ[\ell]\circ\alpha(\Psi)\]
defines a group isomorphism onto $G$.
\end{prop}
\begin{proof}
By the previous lemma, $b$ indeed commutes with $\alpha(\Psi)$ for $\Psi\in \DAut_{st}(X)$. Together with 
theorem \ref{Plo} and the fact that shifts commute with every derived automorphism, this shows that $S$ is indeed a well-defined group homomorphism with image $G$. Now consider $g=b^k\circ [\ell]\circ \alpha(\phi_*)\circ \alpha(M_L)$ and assume $g=\id$. 
For every point $[\xi]\in X^{[n]}\setminus X^{[n]}_\Delta$ we have $g(\C([\xi]))=\C([\phi(\xi)])[\ell]$ which shows $\ell=0$ and $\phi=\id$, i.e.\ $g=b^k\circ M_{L^{\boxtimes n}}$. Hence, for $A\in \D^b(X)$ its image under $F$ gets mapped to 
$g(FA)=F(A\otimes N)[k(2n-2)]$ for some line bundle $N$ on $X$, which shows that $k=0$. Finally, $g=M_{L^{\boxtimes n}}$ is trivial only if $L=\reg_X$.   
\end{proof}
\begin{remark}
Again, the analogous statement with $b$ replaced by $\sqrt b$ holds.  
\end{remark}
Let now $X$ be a K3-surface. In this case Addington has shown in \cite{Add} that the Fourier--Mukai transform $F_a\colon \D^b(X)\to \D^b(X^{[n]})$ with kernel the universal sheaf $\I_{\Xi}$ is a $\P^{n-1}$ functor with $H=[-2]$. 
Here, $\Xi\subset X\times X^{[n]}$ is the universal family of length $n$ subschemes.
We denote the associated $\P^{n-1}$-twist by $a$ and in case $n=2$ the spherical twist by $\sqrt a$. 
\begin{lemma}
For every point $[\xi]\in X^{[n]}\setminus \partial X^{[n]}$, i.e.\ $\xi=\{x_1,\dots,x_n\}_{\text{red}}$ with pairwise distinct $x_i$,
the object $a(\C([\xi]))$ is supported on the whole $X^{[n]}$. In case $n=2$ the same holds for the object $\sqrt a(\C([\xi]))$.
\end{lemma}
\begin{proof}
We set for short $A=\C([\xi])$.
We will use the exact triangle of Fourier--Mukai transforms $F\to F'\to F''$ with kernels $\mathcal P'=\reg_{X\times X^{[n]}}$ and $\mathcal P''=\reg_{\Xi}$. The right adjoints form the exact triangle $R''\to R'\to R$ with kernels $\mathcal Q''=\reg_{\Xi}^\vee[2]$ and $\mathcal Q'=\reg_{X\times X^{[n]}}[2]$. Over the open subset $X^{[n]}\setminus \partial X^{[n]}$, 
the universal family $\Xi$ is smooth and thus on $\Xi_{|X^{[n]}\setminus \partial X^{[n]}}$ the object $\reg_\Xi^\vee$ is a line bundle  concentrated in degree 2. This yields
\[R''(A)=\reg_\xi[0]\quad,\quad R'(A)=\Ho^*(X^{[n]},A)\otimes \reg_X[2]=\reg_X[2]\,.\]
Setting $H^i=\sHo^i(R(A))$ the long exact cohomology sequence gives $H^{-2}=\reg_X$, $H^{-1}=\reg_\xi$, and $H^i=0$ for all other values of $i$. The only functor in the composition $F=\pr_{X^{[n]}*}\left(\pr_X^*(\_)\otimes \I _\Xi\right)$ that needs to be derived is the push-forward along $\pr_{X^{[n]}}$. The reason is that the non-derived functors $\pr_X^*$ as well as $\pr_X^*(\_)\otimes \reg_\Xi$ are exact (see \cite[proposition 2.3]{Sca2} for the latter). Thus, there is the spectral sequence 
\[E_2^{p,q}=\sHo^p(F(H^q))\Longrightarrow E^n=\sHo^n(FR(A))\]
associated to the derived functor $\pr_{X^{[n]}*}$. It is zero outside of the $-1$ and $-2$ row. Now $F'(\reg_\xi)=\Ho^*(X,\reg_\xi)\otimes \reg_{X^{[n]}}=\reg_{X^{[n]}}^{\oplus n}[0]$ and
$F''(\reg_\xi)$ is also concentrated in degree zero since $\Xi$ is finite over $X^{[n]}$. By the long exact sequence we see that all terms in the $-1$ row except for $E_2^{0,-1}$ and $E_2^{1,-1}$ must vanish. Furthermore, 
\[F'(H^{-2})=\Ho^*(X,\reg_X)\otimes \reg_{X^{[n]}}= \reg_{X^{[n]}}[0]\oplus \reg_{X^{[n]}}[-2]\] and $F''(H^{-2})$ is a locally free sheaf of rank $n$ concentrated in degree zero since $\Xi$ is flat of degree $n$ over $X^{[n]}$. This shows that the $-2$ row of $E_2$ is zero outside of degree $0$, $1$, and $2$ and that $E^{1,-2}_2$ is of positive rank. By the positioning of the non-zero terms it follows that $E^{1,-2}_2=E_\infty^{1,-2}$ and thus also $E^{-1}=\sHo^{-1}(FR(A))$ is of positive rank. Furthermore, we can read off the spectral sequence that the cohomology of $FR(A)$ is concentrated in the degrees $-2$, $-1$, and $0$.
Now, by the long exact sequences associated to the cones occurring in the definition of the spherical respectively the $\P^n$-twist 
it follows that $\sHo^{-2}(\sqrt a(A))$ as well as $\sHo^{-2}(a(A))$ are of positive rank.   
\end{proof}
\begin{prop}
\begin{enumerate}
 \item  The subgroup $H$ generated by $a$ and push-forwards along natural automorphisms, i.e.\ autoequivalences of the form $\phi^{[n]}_*=\alpha(\phi_*)$, is isomorphic to $\Z\times \Aut(X)$.
\item $b\notin H=\langle a,\{\phi^{[n]}_*\}_{\phi\in \Aut(X)}\rangle$.
\item $a\notin G=\langle b,[\ell],\alpha(\DAut_{st}(X))\rangle$.
\end{enumerate}
The same results hold for $a$ replaced by $\sqrt a$ and $b$ replaced by $\sqrt b$. 
\end{prop}
\begin{proof}
We have for $\phi\in \Aut(X)$ that $\phi^{[n]}_*\circ F_a=F_a\circ \phi_*$ which by lemma \ref{relation} shows that $a$ commutes with $\phi^{[n]}_*$. The reason is that the subvariety $\Xi\subset X\times X^{[n]}$ is invariant under the morphism $\phi\times \phi^{[n]}$. Because of $ a^k\circ \phi^{[n]}_*(F_a(A))=F_a(\phi_*^kA)[k2(n-1)]$ for $A\in \D^b(X)$, there are no further relations in the group $H$ which shows (i). The autoequivalence $g=a^k\circ \phi^{[n]}_*\in H$ has $g(F(\reg_X))=F(\reg_X)[2k(n-1)]$. Thus, by 
remark \ref{noshift} the equality $b=g$ implies $k=1$. But also $b=a\circ\phi^{[n]}_*$ can not hold comparing the values of both sides on $\C([\xi])$ for $[\xi]\in X^{[n]}\setminus \partial X^{[n]}$.    
The assertion (iii) also is shown by comparing the values of the autoequivalences on $\C([\xi])$. 
\end{proof}
Using the same arguments as in \cite[p.11-12 and p.39-40]{Add} one can also show that $b$ does not equal a shift of an autoequivalence induced by a $\P^n$-object on $X^{[n]}$ or of an autoequivalence of the form $\alpha(T_E)$ for a spherical twist $T_E$ on the surface.
In particular, $b$ is an exotic autoequivalence in the sense of \cite{PS}.
\section{$\P^n$-objects on generalised Kummer varieties}
Let $A$ be an abelian surface. There is the summation map 
\[\Sigma\colon A^n\to A \quad,\quad (a_1,\dots,a_n)\mapsto \sum_{i=1}^n a_i\,.\]
We set $N_{n-1}A:=\Sigma^{-1}(0)$. It is isomorphic to $A^{n-1}$ via e.g.\ the morphism
\[A^{n-1}\to N_{n-1}A\quad,\quad (a_1,\dots,a_{n-1})\mapsto (a_1,\dots,a_{n-1},-\sum_{i=1}^{n-1} a_i)\,.\]
The subvariety $N_{n-1}A\subset A^n$ is $\sym_n$-invariant. Thus, we have $N_{n-1}A/{\sym_n}\subset S^nA$.
The \textit{generalised Kummer variety} is defined as $K_{n-1}A:=\mu^{-1}(N_{n-1}A/{\sym_n})$, i.e.\ it is the subvariety of the Hilbert scheme $A^{[n]}$ consisting of all points representing subschemes whose weighted support adds up to zero.
It can be identified with $\Hilb^{\sym_n}(N_{n-1}A)$ and also the other assumptions of the Bridgeland--King--Reid theorem are satisfied which leads to the equivalence 
\[\Phi=\FM_{\reg_{\overline {I^nA}}}\colon \D^b(K_{n-1}A)\to \D^b_{\sym_n}(N_{n-1}A)\]
where $\overline {I^nX}=p^{-1}(N_{n-1}A)$ (see e.g.\ \cite{Nam}). The intersection between the small diagonal $\Delta=\delta(A)\subset A^n$ and $N_{n-1}A$ consists exactly of the points $\delta(a)=(a,\dots,a)$ for $a$ an $n$-torsion point of $A$, i.e.\ $\Delta\cap N_{n-1}A=\delta(A_n)$.
The intersection is transversal, since under the identification $T_{A^n}\cong T_A^{\oplus n}$ the tangent space of $\Delta$ in a point $\delta(a)$ with $a\in A_n$ is given by vectors of the form $(v,\dots,v)\in T_A(a)^{\oplus n}$ whereas the tangent space of $N_{n-1}A$ is given by vectors $(v_1,\dots,v_n)\in T_A(a)^{\oplus n}$ with $\sum_{i=1}^n v_i=0$. Thus, we have for the tangent space of $N_{n-1}A$ in $\delta(a)$ the identification $T_{N_{n-1}A}(\delta(a))\cong N_{\Delta/A^n}(\delta(a))$. Since the $\sym_n$-action on $N_{n-1}A$ is just the restriction of the action on $A^n$, this isomorphism is equivariant.
\begin{theorem}
Let $n\ge 2$. For every $n$-torsion point $a\in A_n$ the skyscraper sheaf $\C(\delta(a))$ is a $\P^n$-object in 
$\D^b_{\sym_n}(N_{n-1}A)$.
\end{theorem}
\begin{proof}
Indeed, using the results for the invariants of $\wedge^* N_{\Delta/A^n}$ of section \ref{dia}
\begin{align*}\Hom^*_{\D^b_{\sym_n}}(\C(\delta(a)),\C(\delta(a)))\cong \Ext^*(\C(\delta(a)),\C(\delta(a)))^{\sym_n}&\cong 
 \wedge^* T_{N_{n-1}A}(\delta(a))^{\sym_n}\\&\cong\wedge^* N_{\Delta/A^n}(\delta(a))^{\sym_n}\\&\cong \C\oplus\C[-2]\oplus\dots\oplus\C[-2n]\,.
\end{align*}
\end{proof}
\begin{remark}
For two different $n$-torsion points the skyscraper sheaves are orthogonal which makes the associated twists commute.
Thus, we have an inclusion (see corollary \ref{pnobj})
\[\Z^{n^4}\subset \Aut(\D^b_{\sym_n}(N_{n-1}A))\cong \Aut(\D^b(K_{n-1}A))\,.\] 
\end{remark}
In the case $n=2$ the generalised Kummer variety $K_{n-1}A=K_1A$ is just the Kummer surface $K(A)$. Moreover, there is an isomorphism of commutative diagrams
\[
\begin{CD}
\overline{I^nA}
@>{p}>>
N_1A \\
@V{q}VV
@VV{\pi}V \\
K_1A
@>>{\mu}>
N_1A/{\sym_2}
\end{CD} 
\quad\cong\quad
\begin{CD}
\tilde A
@>{p}>>
A\\
@V{q}VV
@VV{\pi}V\\
K(A)
@>>{\mu}>
A/\iota
\end{CD} 
\]
where $p$ and $\mu$ in the right-hand diagram are the blow-ups of the 16 different 2-torsion points respectively of their image under the quotient under the involution $\iota=(-1)$. For a 2-torsion point $a\in A_2$ we denote by $E(a)$ the exceptional divisor over the 
point $[a]\in A/\iota$. Since $E(a)$ is a rational curve in the $K3$-surface $K(A)$, every line bundle on it is a spherical object
in $\D^b(K(A))$. 
\begin{prop}
 For every 2-torsion point $\Phi^{-1}(\C(\delta(a)))=\reg_{E(a)}(-2)[1]$ holds.
\end{prop}
\begin{proof}
Using the isomorphism of the commutative diagrams above the proof is nearly the same as the proof of 
proposition \ref{underbkr}.
\end{proof}
There is no known homomorphism $\Aut(\D^b(A))\to \Aut(\D^b(K_nA))$ analogous to Ploog's map $\alpha$. But at least one can 
lift line bundles $L\in \Pic A$ (by restricting $\mathcal D_L$) and group automorphisms $\phi\in \Aut(A)$ (by restricting $\phi^{[n]}$) to 
the generalised Kummer variety. Recently, Meachan has shown in \cite{Mea} that the restriction of Addington's functor to the generalised Kummer variety $K_{n}(A)$ for $n\ge 2$  (i.e.\ the Fourier--Mukai transform with kernel the universal sheaf) is
still a $\P^{n-1}$ functor and thus yields an autoequivalence $\bar a$. Comparing these autoequivalences with those induced by the 
above $\P^n$-objects one gets results similar to the results of section \ref{comp}.

\bibliographystyle{alpha}
\addcontentsline{toc}{chapter}{References}
\bibliography{references}

\end{document}